\documentclass[12pt]{article}
\usepackage{latexsym}
\usepackage[doublespacing]{setspace}
\usepackage{tocloft}
\usepackage{graphicx}
\usepackage{graphicx,epstopdf}
\usepackage{amsmath}
\usepackage{amssymb}
\usepackage{amsthm} % Jim added this.  It makes proof environment work.
\usepackage{algorithm}
\usepackage{algorithmic}
\usepackage{xcolor}
\usepackage[export]{adjustbox}
\usepackage{tikz}

\newcommand{\comment}[1]{} %define a new command that effectively does nothing with the input

\usepackage{caption}
\usepackage{float}
\usepackage{enumerate}
\usepackage{hyperref}

\newtheorem{theorem}{Theorem}
\newtheorem{lemma}[theorem]{Lemma}
\newtheorem{example}[theorem]{Example}
\newtheorem{remark}[theorem]{Remark}

\newtheorem{definition}[theorem]{Definition}

\newcommand{\tf}{\tilde f}

\newcommand{\lam}{\lambda}

\newcommand{\eps}{\varepsilon}

\newcommand{\R}{\mathcal R}

\newcommand{\RL}{\mathbb R}
\newcommand{\Rn}{\RL^n}

\input{epsf}
\newcommand{\newold}[2]{{\bf new:}{\color{blue}#1}{\bf old:}{\color{red}#2}}%
\newcommand{\note}[1]{{\color{teal}{\bf Note:} #1}}%

\begin{document}

\title{
On Second-Order Cone Functions
}

\author{Shafiu Jibrin \& James W. Swift\\
Department of Mathematics and Statistics\\
Northern Arizona University\\
Flagstaff, AZ 86011-5717, USA}

\date{\today}

\maketitle

\begin{abstract}
%\begin{singlespace}
We consider the {\it second-order cone function} (SOCF) $f: \RL^n \to \RL$
defined by $f(x)= c^T x + d -\|A x + b \|$, with parameters
$c \in \RL^n$, $d \in \RL$, $A \in \RL^{m \times n}$, and $b \in \RL^m$.
Every SOCF is concave.  We give
necessary and sufficient conditions for strict concavity of $f$. 
The parameters $A$ and $b$ are not uniquely determined.
We show that every SOCF can be written in the 
form
$f(x) = c^T x + d -\sqrt{\delta^2 + (x-x_*)^TM(x-x_*)}$.
We give necessary and sufficient conditions for the parameters $c$, $d$, $\delta$,
$M = A^T A$, and $x_*$ to be uniquely determined.
We also give necessary and sufficient conditions for $f$ to be bounded above.
%\end{singlespace}
\end{abstract}

\vfill

\noindent
\emph{2000 Mathematics Subject Classification:} 90C22, 90C25, 90C51.
\vspace{0.3cm}

\noindent
\emph{Keywords and phrases:}
Second-order cone constraint, second-order cone programming, second-order cone function, convexity, optimization, 
interior-point methods.

\vspace{0.3cm}
\noindent
\emph{email address:} Jim.Swift@nau.edu, Shafiu.Jibrin@nau.edu

\vfill

\setcounter{equation}{0}
\pagebreak
\section{Introduction}
Second-order cone programming is an important convex optimization problem \cite{Van98, Vand98, Zhao21, SOCPwiki}.
A second-order 
cone constraint has the form $\|A x + b\|\le c^T x + d $,
where $\|\cdot\|$ is the Euclidean norm.
This second-order cone constraint is equivalent to the inequality $f(x) \geq 0$, where $f$ is 
what we call a second-order cone function.
The solution set of the constraint is convex, and the function
$f$ is concave \cite{Van98, Weig10}.  

In the following definition, we use $\RL$ to denote the set of real numbers and 
$\RL^{m \times n}$ to denote the set of $m \times n$ matrices with real entries.  Of course, $m$ and $n$ are positive integers.
\begin{definition}
\label{SOCFdef}
A \emph{second-order cone function (SOCF)} is a function $f: \RL^n \to \RL$ that can be written as
\begin{equation}
\label{socfOld}
f(x)=  c^T x + d -\|A x + b \|
\end{equation}
%where $\| \cdot \|$ is the Euclidean norm,
with parameters $c \in \RL^n$,  $d \in \RL$, $A \in \RL^{m \times n}$, and $b \in \RL^m$.
\end{definition}
%The inequality $f(x)=  c^T x + d -\|A x + b\| \geq 0$ is equivalent to the  second-order cone constraint $\|A x + b\|\le c^T x + d $.
In second-order cone programming 
%\cite{Vand98, SOCPwiki} 
a linear function of $x$ is minimized subject to one or more second-order cone constraints, along with the constraint
$F x = g$, where $F \in \RL^{p \times n}$ and $g \in \RL^p$.
The solution set of $F x = g$ is an affine subspace, and we will show that
the restriction of an SOCF to an affine subspace 
is another SOCF.  Thus, from a mathematical point of view the constraint $F x = g$
is not necessary, although in applications it can be convenient.
In this paper we do not consider the constraint
$Fx = g$ but instead focus on understanding the family of SOCFs.

%Many convex optimization problems such as second-order cone programming problems optimize a function subject to a system of second-order cone constraints \cite{Van98, Vand98, Zhao21}.  

There are interior-point methods for solving second-order cone programming problems. 
These methods usually use SOCFs to
impose the second-order cone constraints \cite{Pot20, Vand98, Weig10, Dawu21}.  
%These methods for solving such problems usually turn the second-order cone constraints into second-order cone functions \cite{Pot20, Vand98, Weig10, Dawu21}
Solvers for second-order cone programming problems include CVXOPT and MATLAB \cite{SOCPsol, And23}. 
The study of second-order cone programming and its applications has continued to generate interest for over three decades \cite{Zhang16, Faw19, Zhao21, Luo22, Gil22, Tang22, Chowdhury23}. 

The current research was started to get a deeper
understanding of SOCFs to
improve interior-point algorithms for finding the
weighted analytic center of a system of second-order cone constraints 
\cite{Aghili22,Borg10, Dawu21}.
The current work can lead to improved algorithms.

In this paper, we give a thorough description of the family of SOCFs.  In
the form of Equation~\eqref{socfOld}, 
%the function is smooth except on the the solution set of $Ax+b=0$.
the parameters $A$ and $b$ are not uniquely determined,
since $\| A x + b \| = \| Q(Ax+b)\| = \|(QA) x + (Qb) \|$ for any orthogonal $m \times m$
matrix $Q$.
We show that every SOCF can be written in the 
form
\[f(x) = c^T x + d -\sqrt{\delta^2 + (x-x_*)^TM(x-x_*)},\]
with the parameters $\delta \geq 0$, $x_* \in \Rn$, and the positive semidefinite
$M = A^T A \in \RL^{n \times n}$ replace the parameters $A$ and $b$.
We show that these new parameters are unique if and only if $M$ is positive definite.

It is known that every SOCF $f$ is concave \cite{Van98, Weig10}.
We show that $f$ is strictly concave 
if and only if $\text{rank}(A) = n$ and $b \not \in \text{col}(A)$,
where $\text{col}(A)$ denotes the column space of $A$.
In terms of the new parameters, the SOCF is strictly concave
if and only if $M$ is positive definite and $\delta > 0$.

In the case where $M$ is positive definite, we show that $f$ is bounded above if and only if $c^T M^{-1}c \leq 1$. 
We show that the convex set $\{ x \in \Rn \mid f(x) \geq 0\}$ is bounded if and only
if $M$ is positive definite and $c^T M^{-1}c < 1$. 

Our results have computational implications for convex optimization problems involving second-order constraints such as the problem of minimizing weighted barrier functions presented in \cite{Borg10, Aghili22}.  This is related to the problem of finding a weighted analytic center for second-order cone constraints given in \cite{Dawu21}. There are also computational implications for the problem of computing the region of weighted analytic centers of a system of several second-order cone constraints. This is under investigation as part of our current research which is an extension of the work given in \cite{Dawu21}. 

In the problems presented in \cite{Borg10, Aghili22, Dawu21},
the boundedness of the feasible region guarantees the existence of a minimizer,
and the strict convexity of the barrier function guarantees the uniqueness of the minimizer. Also, the strict convexity of the barrier function affects how quickly we can find the minimizer using these algorithms. The determination of the strict concavity of $f$ is related to the strict convexity of the barrier function. The boundedness of the feasible region of the SOC constraints system is also related to the boundedness of $f$. If a single $f$ is bounded, then the feasible region of the SOC constraints system is bounded.

Convex optimization algorithms perform well and more efficiently when the problem is known to be bounded and the objective function is strictly convex. If a second-order cone function is strictly concave, its gradient and Hessian matrix are defined, and the Hessian is invertible. 
The corresponding barrier function is similarly well-behaved,
and Newton's method and Newton-based methods work well for the problem. 
However, many optimization problems are not bounded or have objective functions that are not strictly convex. Our results would allow one to recognize convex optimization problems involving second-order cone constraints (as in \cite{Borg10, Aghili22, Dawu21}) that can be solved efficiently, or to assist in reformulating those that are hard to solve.

%The performance of the algorithms depends on the boundedness and convexity.
%In this work, the boundedness of the set  $\{ x \in \Rn \mid f(x) \geq 0\}$ %guarantees the existence of a minimizer. 
%The boundedness of the individual constraints affects the number of iterations %and solution times of the algorithms.

\setcounter{equation}{0}
\section{Properties of Second-Order Cone Functions
}

The SOCFs on $\RL$ (that is, $n = 1$) are the simplest to understand,
and give insight into the general case.

\begin{example}
\label{SOCFonR1}
Consider $f: \RL \to \RL$ defined by Equation(\ref{socfOld}) with 
$A = \left [ \begin{smallmatrix}1 \\0 \end{smallmatrix} \right ]$ and $b = \left[ \begin{smallmatrix}-x_* \\ \delta \end{smallmatrix} \right ]$.  Thus $A x + b =  \left[ \begin{smallmatrix}x-x_* \\ \delta \end{smallmatrix} \right ] $, and
$\| Ax+b\| = \sqrt{\delta^2 + (x-x_*)^2}$, so
$f(x) =  cx + d - \sqrt{\delta^2 + (x-x_*)^2}$. 
Figure~\ref{fig:1D}
shows several graphs with various values of the real parameters $\delta$, $x_*$, $c$, and $d$.
If $\delta \neq 0$, then $f$ is smooth and strictly concave, as shown by the solid graphs.  
If $\delta = 0$, then  $f(x) = dx + d - |x-x_*|$ is piecewise linear with a corner at $(x_*, c x_* + d)$, as shown by the dashed graphs.
Note that $f(x_*) = c x_* +d - |\delta|$ for any value of $\delta$, so the solid
graphs in Figure~\ref{fig:1D} (with $\delta = 0.2$) pass a distance 0.2 below the corner of the dashed graphs (with $\delta = 0$), as indicated by the double arrows.
%For any value of $a$, $f(x) = -|x-x_*| + cx + d + \text{O}(1/x)$ as $x \to \pm \infty$.
%If $a = 0$, then $f(x) = -| x- x_*| + cx + d$ is piecewise linear, and if $a \neq 0$ the graph of $f$ has slant asymptotes given by $y = - | x- x_*| + cx + d$.
%which is the graph of the SOCF with the same $A$, $c$, and $d$, but with $b = \left[ \begin{smallmatrix}-x_* \\0 \end{smallmatrix} \right ]$.  
\begin{figure}
    \begin{center}
    \begin{tikzpicture}
        \node at (-5.1,0) {\includegraphics[width = 4.4cm]{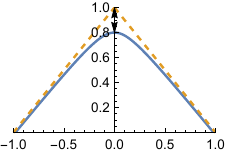}};
        \node at (0,0) {\includegraphics[width = 4.4cm]{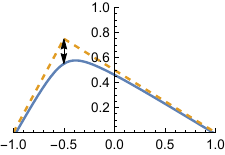}};
        \node at (5.1,0) {\includegraphics[width = 4.4cm]{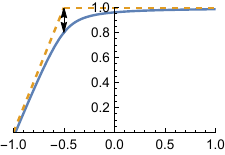}};
        \node at (-5.1, -2) {\small $c = 0, \ d = 1, x_* = 0$};
        \node at (0, -2) {\small$ { c = 0.5, \ d = 1, \ x_* = -0.5 }$};
        \node at (5.1, -2) {\small $c = 1, \ d = 1.5, \ x_* = -0.5$};
    \end{tikzpicture}
    \end{center}
    % source dropbox/Jim and Shafiu 2022/mathematica/socf-n1-m2.nb
    \caption{
     \label{fig:1D}
     Graphs of second-order cone functions $f(x) =  cx + d- \sqrt{\delta^2 + (x-x_*)^2}$, as described in
    Example \ref{SOCFonR1}.
    In each of the three plots, the parameters $c$, $d$, and $x_*$ are indicated. The dashed curve has $\delta = 0$, and the solid curve has $\delta = 0.2$.}
   
\end{figure}
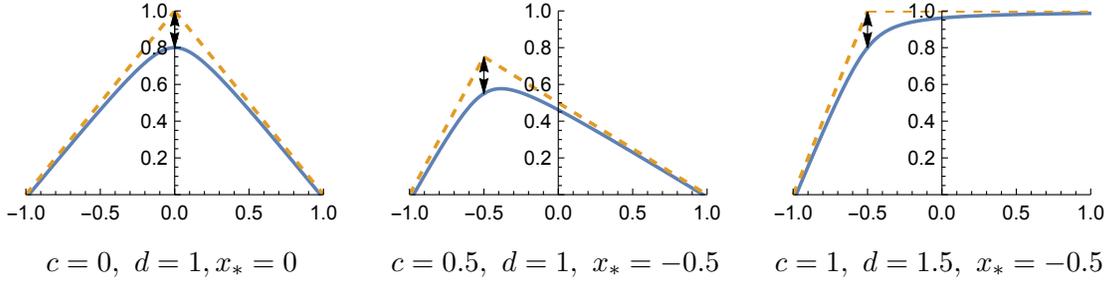
\end{example}
One important property of SOCFs is that their restriction to an affine subspace is another
SOCF. We will frequently restrict  to a 1-dimensional affine subspace.
\begin{remark}
\label{restrictionToAffine}
    Let $f: \Rn \to \RL$ be  written in the form of Equation~(\ref{socfOld}).  
    The restriction of $f$
    to the affine subspace $\{x_0 + B y\mid y \in \RL^k\}$, for some $x_0 \in \Rn$ and
    $B \in \RL^{n \times k}$ is
    $$f(x_0 +B y) = (c^T B) y + (c^T x_0 + d) - \| (A B) y + (A x_0 + b) \|,$$
    which is an SOCF on $\RL^k$ with the variable $y$.
\end{remark}

Recall that a function $f: \RL^n \to \RL$ is \emph{concave}
provided that $f((1-t)x_0 + t x_1) \geq (1-t) f(x_0) + t f(x_1)$
for all $x_0 \neq x_1 \in \RL^n$, and all $t \in (0,1)$.  The function
is \emph{strictly concave} if the inequality is strict. A twice differentiable function $f: \RL \to \RL$ is concave if $f''(x) \leq 0$ for all $x$, and strictly concave if the inequality is strict.
\begin{lemma}
\label{SOCF1D}
    Let $f: \RL \to \RL$ be the general SOCF of one variable, defined by $f(x) = c x + d - \| Ax + b \|$ with parameters $c, d \in \RL$ and $A, b \in \RL^m$.
    The function $f$ is concave for all parameters, and $f$ is strictly concave if and only if $A \neq 0$ and $b \not \in \text{col}(A)$.
\end{lemma}
\begin{proof}
If $A = 0$, then $f(x) = cx +d - \|b\|$ is linear, and hence concave but not strictly concave.
\begin{figure}
\begin{center}
\begin{tikzpicture}
\draw (0,0) -- (6,3) node[anchor=190] {$\text{col}(A)$};
\draw[very thick, ->] (.5,.25) -- (2,1);
\fill[black](.5, .25) circle (.06cm);
\draw (1.8,.5) node {$A$};
\draw [very thick] (3, 1.5) -- (5.5, 2.75) -- (2.1, 3.3) -- cycle;
%\draw (3.2, 1.6) -- (3.1, 1.8) -- (2.9, 1.7) ;
\draw (3.3, 1.65) -- (3.15, 1.95) -- (2.85, 1.8) ;
\fill[black](3, 1.5) circle (.06cm);
\draw (3.3, 1.1) node {$A x_*$};
\fill[black](5.5, 2.75) circle (.06cm);
\draw (4.9, 1.6) node {$\| A(x-x_*)\|$};
\draw (5.7, 2.4) node {$A x$};
\fill[black](2.1, 3.3) circle (.06cm);
\draw (1.80, 3.65) node {$-b$};
\draw (2.3, 2.3) node {$\delta$};
\draw (4.0, 3.4) node {$\|Ax + b\|$};
\end{tikzpicture}
\end{center}
\caption{
\label{1Dproof}
The geometry of an SOCF on $\RL$.  In this case $A \in \RL^2 = \RL^{2\times 1}$.
Note that $A x_*$ is the point in $\text{col}(A)$ that is closest to $-b$.
See Lemma~\ref{SOCF1D}.
}
\end{figure}
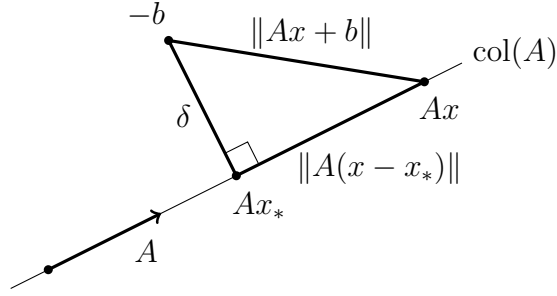
Assume $A \neq 0$.  Then $A x_*$, where $x_* = -(A^T  b)/(A^T A)$, is the point in $\text{col}(A) = \text{span}(A)$ closest to $-b$.
Let $\delta = \| A x_* + b\|$ be the distance from $A x_*$ to $-b$.  Thus 
$\|Ax + b\|^2 = \delta^2 + \|A(x-x_*)\|^2$ by the Pythagorean theorem, and
$f(x) = c x + d - \sqrt{\delta^2 + \|A(x-x_*)\|^2} = cx +d - \sqrt{\delta^2 + (A^T A) (x-x_*)^2}$.
The constant $A^TA$ is a positive real number.
The geometry is shown in Figure~\ref{1Dproof}.
Note that $\delta = 0$ if and only if $b \in \text{col}(A)$.
If $\delta = 0$, then $f(x) = c x + d - \sqrt{A^TA} |x-x^*|$ is piecewise linear with a downward bend at $x_*$, and hence concave but not strictly concave.

So far, we have proved that $f$ is concave but not strictly concave if $A = 0$ or $b \in \text{col}(A)$.

Assume that $A\neq 0$ and $b \not \in \text{col}(A)$.  Then $\delta > 0$, and $f$ strictly concave, since
$$
f''(x) = \frac{-\delta^2 A^T A}{\big(\delta^2 + A^TA (x-x_*)^2\big)^{3/2} }
$$
is defined and negative for all $x$.
\end{proof}
\begin{theorem}
\label{strictlyConcave}
Every second-order cone function $f$ is concave. Furthermore, $f$ is strictly concave if and only if
$\text{rank}(A) = n$ and $b\not \in \text{col}(A)$, using the parameters in
Definition~\ref{SOCFdef}.
\end{theorem}
\begin{proof}
    Let $x_0 \neq x_1 \in \RL^n$, and define $v = x_1-x_0$.  
    Let $g: \RL \to \RL$ be defined by 
    $g(t) := f((1-t) x_0 + t x_1) = f(x_0 + t v) = 
    c^T(x_0 + t v) + d -\| A(x_0 + t v) + b\|$.
    It follows directly from the definition that $f$ is (strictly) concave if and only if
    $g$ is (strictly) concave for all $x_0 \neq x_1$.
    Note that $Av \in \RL^m$.  If $Av = 0$, then $g$ is linear.  If $Av \neq 0$, then
    $$
    g(t) = \tilde c \, t + \tilde d - \sqrt{ \tilde \delta^2 + \|Av\|^2 (t - t_*)^2},
    $$
      where $\tilde c = c^T v$, $\tilde d = c^Tx_0 + d$, 
      $\tilde \delta = \|A(x_0 + t_* v) + b\|$, and $t_* = \frac{-(Av)^T(A x_0 + b)}{(Av)^T Av}$ are all real numbers.
    Thus, $g$ is a second-order cone function of one variable.  By Lemma~\ref{SOCF1D}, $g$ is concave for all choices
    of $x_0$ and $x_1$, and hence $f$ is concave.

    Since $A \in \RL^{m \times n}$, it follows that $\text{rank}(A) \leq n$.  If $\text{rank} (A) < n$ then $A^T A$ is singular, and there exist $x_0 \neq x_1 = x_0 + v$ such that $A v = 0$ and hence $g$ is linear.  If $b \in \text{col}(A)$ then there exist $x_0$ such that $A x_0 + b = 0$.  Thus $t_* = 0$ and $\tilde \delta = 0$, and $g$ is piecewise linear with a downward corner. Thus, if $\text{rank}(A) < n$ or $b \in \text{col}(A)$ (or both), we can find $x_0 \neq x_1$ such that $g$ is concave but not strictly concave, and hence $f$ is not strictly concave.

    Now assume $\text{rank}(A) = n$ and $b \not \in \text{col}(A)$.  It follows that $A v \neq 0$ and $\tilde \delta > 0$ for all $x_0 \neq x_1$.
    Lemma~\ref{SOCF1D} implies that $g$
    is strictly concave for all $x_0 \neq x_1$, 
    and it follows that $f$ is strictly concave.
\end{proof}

Note that $A \in \RL^{n \times n}$ cannot satisfy $\text{rank}(A) = n$ and $b \not \in \text{col}(A)$.
Therefore, any SOCF with $A \in \RL^{n \times n}$ is concave but not strictly concave.  

\begin{figure}
\begin{center}
 \begin{tabular}{cc}
 % figures were produced by Jim and Shafiu 2023/mathematica/SOCFsOnR2.nb
 % width was 2.2 inches in original submitted manuscript.
\includegraphics[width = 2.0in]{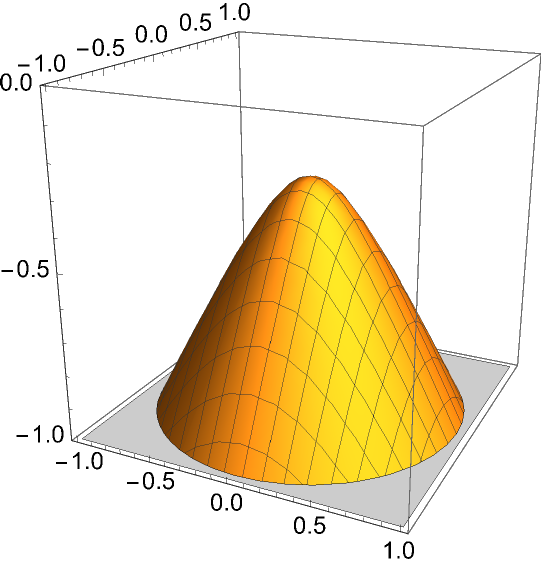} &
\includegraphics[width = 2.0in]{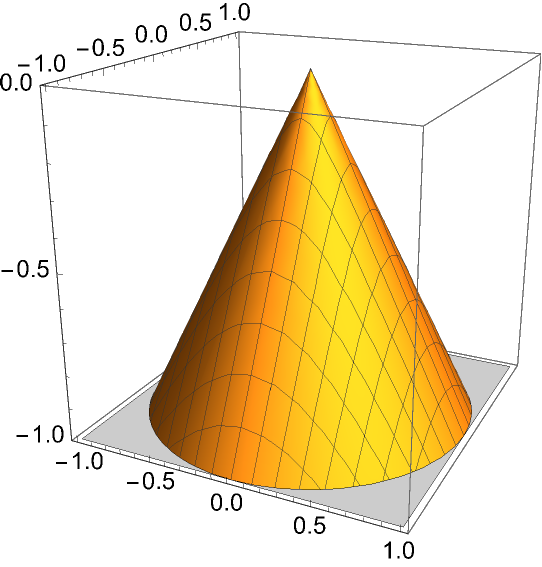}\\
(a) & (b)
\\ \\
\includegraphics[width = 2.0in]{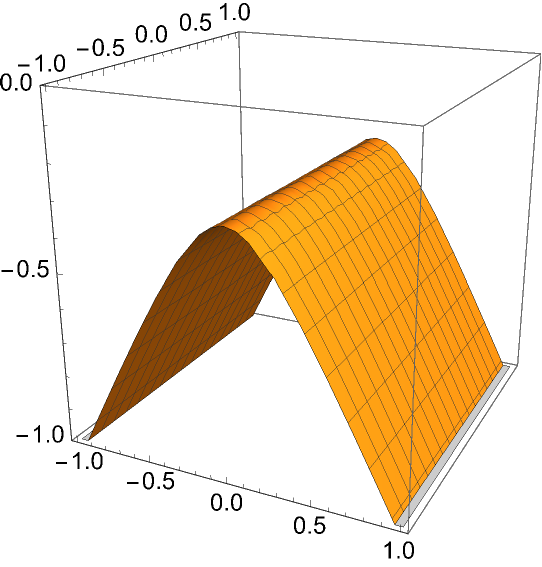} &
\includegraphics[width = 2.0in]{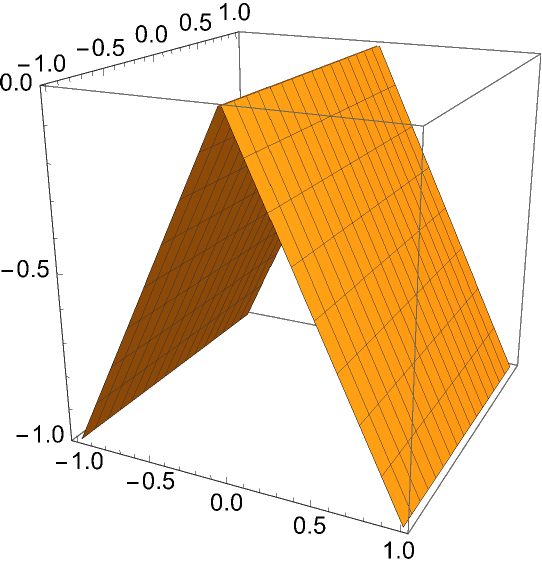} \\
(c) & (d)
\end{tabular}   
\end{center}
\caption{
\label{4SOCFfigs}
Graphs of the four SOCFs on $\RL^2$ defined in Example~\ref{4SOCFexample}.  
Note that the graph of the SOCF (b) is indeed a cone.
The top row shows functions with $\text{rank}(A)  = 2$
and the bottom row shows $\text{rank}(A)  = 1$.
The left column shows $b \not \in \text{col}(A)$ and the right
column shows $b \in \text{col}(A)$.
All the functions graphed are concave, 
but only the upper left function is strictly concave, 
in agreement with Theorem~\ref{strictlyConcave}.
}
\end{figure}

\begin{example}
\label{4SOCFexample}
We give four examples of SOCFs on $\RL^2$,
with different truth values of $\text{rank}(A) = 2$ or $b \in \text{col}(A)$.  
These SOCFs have $c = 0$ and $d = 0$, so $f(x) = - \|Ax+b\|$.
%Let $x = (u,v) \in \RL^2$.  
See Figure~\ref{4SOCFfigs}.

\emph{(a)} $A = \left[ \begin{smallmatrix}
    1 & 0 \\
    0 & 1 \\
    0 & 0
\end{smallmatrix} \right] $ and $b = \left[ \begin{smallmatrix}
    0 \\
    0 \\
    0.3
\end{smallmatrix} \right] $ yields
$f(x) = - \sqrt{0.09 + x_1^2 + x_2^2}$.

\emph{(b)} $A = \left[ \begin{smallmatrix}
    1 & 0 \\
    0 & 1
\end{smallmatrix} \right] $ and $b = \left[ \begin{smallmatrix}
    0 \\
    0
\end{smallmatrix} \right] $ yields
$f(x) = - \sqrt{x_1^2 + x_2^2}$.

\emph{(c)} $A = \left[ \begin{smallmatrix}
    1 & 0 \\
    0 & 0 \\
    0 & 0
\end{smallmatrix} \right] $ and $b = \left[ \begin{smallmatrix}
    0 \\
    0 \\
    0.3
\end{smallmatrix} \right] $ yields
$f(x) = - \sqrt{0.09 + x_1^2}$.

\emph{(d)} $A = \left[ \begin{smallmatrix}
    1 & 0 \\
    0 & 0
\end{smallmatrix} \right] $ and $b = \left[ \begin{smallmatrix}
    0 \\
    0
\end{smallmatrix} \right] $ yields 
$f(x) = -|x_1|$.
\end{example}

Notice that we have frequently rewritten $\|Ax +b\|$ in terms of a square root, as
in Examples~\ref{SOCFonR1} and \ref{4SOCFexample}.  We have also noted that 
$\|Ax +b\| = \| QA x + Qb\|$ for any orthogonal matrix $Q$, so many different choices
of $A$ and $b$ define the same SOCF.  
The next theorem describes a useful way to write an SOCF.

This theorem uses the Moore-Penrose Inverse of a matrix, also called the pseudoinverse, 
which has many interesting properties found in
\cite{MPIwiki}.
For example, $x = A^+b$ is the least squares solution to $Ax = b$, where $A^+ \in \RL^{n \times m}$ is the pseudoinverse of 
$A \in \RL^{m \times n}$.

The next theorem mentions the well-known fact that $A^T A$ is a \emph{positive semidefinite} matrix, which means that it is symmetric with non-negative eigenvalues.  A \emph{positive definite} matrix is a symmetric matrix with all positive eigenvalues.  If $A \in \RL^{m \times n}$ then $A^TA$ is positive definite if and only if the rank of $A$ is $n$.
\begin{theorem}
\label{MdeltaForm}
Every SOCF of the form $f(x) = c^T x + d - \|Ax+b\|$ is identically equal to
\begin{equation}
\label{socf}
f(x) = c^T x + d -\sqrt{\delta^2 + (x-x_*)^T M (x-x_*)},
\end{equation}
where $M = A^T A$ is positive semidefinite, $x_* = -A^+ b$, and $\delta = \| A x_* + b\|$.
\end{theorem}
\begin{proof}
It is well-known that  the least squares solution to 
$A x = -b$ is $x_* = - A^+ b$, and that $A x_* = - A A^+ b$ is the
orthogonal projection of $-b$ onto $\text{col}(A)$.
That is, $A x_*$ is the point in $\text{col}(A)$ that is closest to $-b$.
Thus, the distance squared from $Ax$ to $-b$ is the distance squared from $Ax$ to $A x_*$ plus the distance squared from $A x_*$ to $-b$.
That is,
\begin{align*}
\| Ax + b \|^2 & = \| Ax -Ax_* \|^2 + \| A x_* + b\|^2 \\
               & = \| A(x-x_*) \|^2 + \| A x_* + b\|^2 \\
               & = (x - x_*)^T A^T A (x-x_*) + \| A x_* + b\|^2 \\
               & = (x - x_*)^T M (x-x_*) + \delta^2.
\end{align*}
The last equality uses the definitions of $M$ and $\delta$.  The result follows.
\end{proof}

\begin{remark}
For $A \in \RL^{m \times n}$, note that $\text{rank}(A) = n$ 
if and only if $A^T A \in \RL^{n \times n}$ is positive definite.  
The definition of $\delta$ in Theorem~\ref{MdeltaForm} makes it clear that 
$b \in \text{col}(A)$ if and only if $\delta = 0$.
    Therefore Theorem \ref{strictlyConcave} implies that an SOCF written in the form
    of Equation~(\ref{socf}) is strictly concave 
    if and only if $M$ is positive definite and $\delta > 0$.
\end{remark}
\begin{example}
\label{socf23example}
The left half of Figure \ref{socf-n2-m3} shows the critical point and one contour of the SOCF $f(x) = - \| Ax + b \|$, with 
$$A = \begin{bmatrix} 1&0\\-1&1\\0&2 \end{bmatrix}, \quad
b = \begin{bmatrix} 1\\1\\-1 \end{bmatrix}.
$$
The right part of the same Figure shows the geometry behind Theorem \ref{MdeltaForm}, which describes how to write
the function in the form $f(x) =  -\sqrt{\delta^2 + (x-x_*)^T M (x-x_*)}$.
Calculations show
$$
A^+  = \frac{1}{9}
\begin{bmatrix} 5 & -4 & 2 \\ 1 & 1 & 4
\end{bmatrix}, \quad
M = A^T A =  \begin{bmatrix} 2& -1\\-1 & 5 \end{bmatrix}, \quad
\textstyle \ x_* = \left (\frac 1 9, \frac 2 9 \right ), \quad \text{and} \
 \delta = \frac{5}{3}.
$$
The image of the square in $\RL^2$ under $A$ is the 
light blue parallelogram in $\RL^3$, shown on the right side of Figure~\ref{socf-n2-m3}.  The vectors in 
$\RL^3$ are the first (blue) and second (red) column of $A$.
These span the column space of $A$ in $\RL^3$.
The dot in $\RL^2$ is $x_*$, and the dot in the column space is $A x_* = -A A^+ b = \left (
1/9, 1/9,  4/9 \right )$, which
is the orthogonal projection of $-b$ onto $\text{col}(A)$.
The other dot in $\RL^3$ is $-b$.  The distance from $Ax_*$ to $-b$ is $\delta = 5/3$, so $f(x_*) = -5/3$.  
The ellipse on the left
is the contour of $f$ with height $-2$. 
The image of the ellipse under $A$ is the circle on the right, with is the set of points in the column space that are distance $2$ from $-b$. 
%\newold{}{
%Figure~\ref{socf-n2-m3graph} shows a graph of and contour map of this same function on the left.  The right shows the SOCF with the same parameters except $b = (-1/9, -1/9, -4/9)$, which makes $\delta = 0$.
%}
\end{example}

\begin{figure} %putting [H] really worked!  
%\centerline{\includegraphics[width = 6 in]{socf-n2-m3.pdf}}
% figures made by mathematica/socf-n2-m3.nb
\begin{center}
\begin{tikzpicture}
   \node at (0,0) {\includegraphics[width = 5cm]{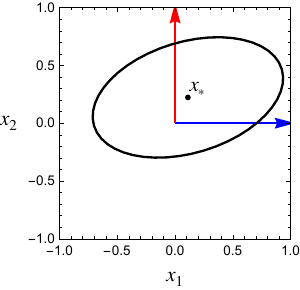}};
    \node at (7,0) {\includegraphics[width = 6cm]{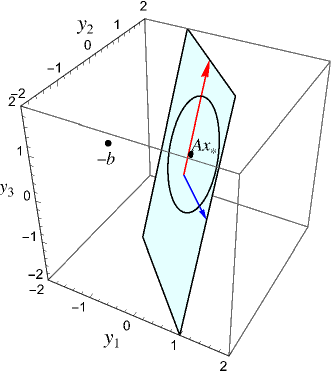}};
\end{tikzpicture}
\end{center}
\caption{The geometry of the second-order cone function 
$f(x) = -\|Ax +b\|$, with $A\in \RL^{3\times 2}$ and $b \in \RL^3$
defined in Example~\ref{socf23example}.
The function can also be written as $f(x) = - \sqrt{\delta^2 + (x-x_*)^T M (x-x_*)}$, were $M = A^T A$.
The maximum of $f$ is at $x_* = -A^+ b$, and the maximum value is $f(x_*) = -\delta$. 
The orthogonal projection of $-b$ onto the column space of $A$ is $A x_* = -AA^+b$.  The distance from $Ax_*$ to $-b$ is $\delta$.
One contour of $f$ is shown.  The image of this contour is a circle of
points in $\text{col}(B)$ that are equidistant from $-b$.
\label{socf-n2-m3}
}
\end{figure}

% Beginning of figure to remove
\comment{
\begin{figure}
\label{SOCFcontoursCone}
\centerline{\includegraphics[height=9.5cm]{SOCFcontoursCone.pdf}}
%Dropbox/Jim and Shafiu 2022/SOCFcontoursCone.pdf
\caption{
\label{socf-n2-m3graph} % Are there any more references to this fig?
\newold{}{
Graph and contour plot of two second-order cone functions 
with $c = 0$, $d = 0$.
%$f(x) = -\sqrt{\delta^2 + (x-x_*)^TM(x-x_*)}$. 
Both functions have 
$M = \left[\begin{smallmatrix}2 &-1\\-1 & 5 \end{smallmatrix} \right ]$
and
$x_* = \left(1/9, 2/9
\right )$.  The function on the left has $\delta = 5/3$ and is the one shown in Figure~\ref{socf-n2-m3}.
The function on the right has $\delta = 0$.  
In each case, 
the contour with $f(x) = -2$ is a thick red curve, and the spacing between contours is $\Delta f = 1/6$.
See Example~\ref{socf23example}.
}
}
\end{figure}
}
% end of commented-out figure

The proof Theorem~\ref{uniqueParams}, to follow, is subtle.  While it is obvious that changing
one parameter will change the function $f$, it is difficult to eliminate the possibility
that more than one parameter can be changed while leaving the function unchanged.
For example, with the form of Equation~\eqref{socfOld}, the function $f$ is unchanged
when $A \mapsto QA$ and $b \mapsto Qb$ for an orthogonal matrix $Q$.
The strategy in the proof is to uniquely determine one parameter at a time 
in a specific order.

\begin{theorem}
\label{uniqueParams}
Assume an SOCF is written in the form of
Equation \eqref{socf}, and that the same SOCF is
written with possibly different parameters satisfying the same
requirements, so 
%that $\tilde M$ is positive semidefinite and $\tilde \delta \geq 0$, etc., so
$$f(x) = c^T x + d - \sqrt{\delta^2 + (x-x_*)^T M (x-x^*)} = \tilde c^T x + \tilde d - 
\sqrt{\tilde \delta^2 + (x-\tilde x_*)^T \tilde M (x-\tilde x^*)}$$
for all $x$.
%If $M$ is positive definite then all of the parameters are
%the same, so the parameterization of $f$ is unique.
%$\tilde d = d$, $\tilde \delta = \delta$, and $\tilde x_* = x_*$ so the parameters are unique.
%More specifically,
\begin{itemize}
    \item If $M = 0$ (the zero matrix), then $\tilde c = c$, $\tilde M = 0$, $\tilde d - \tilde \delta = d - \delta$, and $\tilde x_*$ arbitrary, and
    \item if $M \neq 0$, then $\tilde c = c$, $\tilde d = d$, $\tilde \delta = \delta$, $\tilde M = M$, and $M \tilde x_* = M x_*$.
\end{itemize}
As a consequence, the parameterization of
an SOCF in the form of Equation~\eqref{socf}
is unique if and only if $M$ is positive definite.
\end{theorem}

\begin{proof}
Recall that $M, \tilde M$ are positive semidefinite.
It follows that $Mv = 0$ if and only if $v^T M v = 0$.
Also, recall that $\delta, \tilde \delta$ are non-negative real numbers.

For nonzero $v \in \Rn$ and $t \in [0, \infty)$,
we consider the function
$f(v t)$ and its asymptotic behavior as $t \to \infty$.
If $v^T M v = 0$, then $f(v t) = c^Tv \, t + d - \sqrt{\delta^2 + x_*^T M x_*}$.
If $v^T M v \neq 0$, then
\begin{align}
\nonumber
f(vt) & = c^T v \, t +d - \sqrt{\delta^2 + (v t-x_*)^T M (vt-x_*)} \\
\nonumber & = c^T v \, t +d - \sqrt{v^TMv \, t^2 - 2 v^T M x_* t + x_*^TMx_* + \delta^2  } \\
\nonumber
& = c^T v \, t + d - \sqrt{v^TMv}\, t \, \sqrt{ 1  + \frac {-2 v^T M x_* \, t + x_*^TMx_* + \delta^2}{v^TMv \, t^2} } \\
\label{slant}
&= \left ( c^T v - \sqrt{v^TMv} \right ) t + d  + \frac{v^T M x_*}{\sqrt{v^TMv}} + O(1/t) \text{ as } t \to \infty.
\end{align}
The third equation uses the fact that $t \geq 0$, and the 
fourth equation uses the Taylor series $\sqrt{1 + \eps} = 1 + \eps/2 + O(\eps^2)$
as $\eps \to 0$.  
The fourth equation describes the slant asymptote of the graph of $f(vt)$, and is crucial for the remainder of the proof.

For all $v \neq 0$, Equation~\eqref{slant} implies that
\begin{equation}
\label{fvtDif}
\frac{f(v t) - f(-v t)}{2} = \begin{cases}
    c^T v \, t  & \mbox { if } v^T Mv = 0 \\
    c^T v \, t  + \frac{v^T M x_*}{v^T M v} + O(1/t) & \mbox { if } v^T Mv \neq 0 .
\end{cases}
\end{equation}
A similar expression  where $c$ is replaced by $\tilde c$ holds.
If $v^TMv = 0$, then $\tilde c^Tv = c^Tv$.
If $v^TMv \neq 0$, then
the slope of the slant asymptote is the same for both sets of parameters, so again $\tilde c^T v = c^T v$.  This holds for all $v$, so $\tilde c = c$.

For all $v \neq 0$, Equation~\eqref{slant} implies that
\begin{equation}
\label{fvtSum}
\frac{f(v t) + f(-v t)}{2} = \begin{cases}
    d - \sqrt{\delta^2 + x_*^T M x_*}  & \mbox { if } v^T Mv = 0 \\
    d -\sqrt{v^T M v} \, t   + O(1/t) & \mbox { if } v^T Mv \neq 0 ,
\end{cases}
\end{equation}
along with a similar expression where $d$ is replaced by $\tilde d$, etc.
If $\tilde M \neq M$, then there is some vector 
$v$ such that 
$v^T \tilde M v \neq v^T M v$.  
This leads to a contradiction since the slope
of the slant asymptote 
in Equation~\eqref{fvtSum} would be different.  Thus, $\tilde M = M$.

Assume $M = 0$.
Then $f(x) = c^T x + d - \delta = c^Tx + \tilde d - \tilde \delta$, since $\tilde c = c$, and $\tilde M = M = 0$.
Thus, $\tilde d - \tilde \delta = d - \delta$.

Assume $M \neq 0$.  Then there exists $v \in \Rn$ that satisfies $Mv \neq 0$.  
Using Equation~\eqref{fvtSum} with $v^T \tilde M v = v^T M v \neq 0$,
we find that $\tilde d = d$.
At this point we conclude, from the equality of the two expressions for $f$, that
$\delta^2 + (x-x_*)^T M (x-x_*) = \tilde \delta^2 + (x-\tilde x_*)^T M (x-\tilde x_*)^T$ for all $x$.
Expanding the quadratic term and canceling like terms, we find that
$\delta^2 - 2 x^T M x_* = \tilde \delta^2 - 2 x^T M \tilde x_*$ for all $x$.
Thus $\tilde \delta = \delta$ and $M \tilde x_* = M  x_*$.  

Now we show that the parameterization of $f$ is unique if and only if $M$ is positive definite.
If $M$ is not positive definite there exist $x_* \neq \tilde x_*$ such that 
$M \tilde x_* = M  x_*$.
%
%We have shown that the parameterization is not unique if $M$ is not positive definite.
If $M$ is positive definite then $M \neq 0$ and $M$ is invertible, so $\tilde x_* = x_*$ and all of the parameters are unique.
\comment{ %old proof from here to line 706
\note{Old Proof follows:
Recall that $M$ is positive semidefinite.
Assume $M$ is not positive definite.  
Then $M$ is singular, and the null space of $M$ is not $\{0\}$.
Thus $x_*$ is not unique, since $(x-x_*)^TM(x- x_*) = (x-{\tilde x}_*)^TM(x- \tilde x_*)$ if 
$x_*-\tilde x_*$ is in the null space of $M$.
Thus, the parameterization of $f$ is not unique.

Assume $M$ is positive definite.  
%Consider the function $g(x) = f(x-x_*) =  c^T x + d - \sqrt{\delta^2 + x^T M x }$.
Let $\lam_1 > 0$ be the smallest eigenvalue of $M$, so
$x^T M x \geq \lam_1 \|x\|^2$ for all $x$.
Thus, the Taylor expansion $\sqrt{1 + \eps} = 1 + \eps/2 + O(\eps^2)$ 
as $\eps \to 0$ implies that 
\begin{align*}
f(x) & = c^T x +d - \sqrt{\delta^2 + (x-x_*)^T M (x-x_*)} \\
& = c^T x +d - \sqrt{x^TMx - 2 x^T M x_* + x_*^TMx_* + \delta^2  } \\
& = c^T x + d - \sqrt{x^TMx} \, \sqrt{ 1  + \frac {-2 x^T M x_* + x_*^TMx_* + \delta^2}{x^TMx} } \\
&= c^T x + d - \sqrt{x^TMx} + \frac{x^T M x_*}{\sqrt{x^TMx}} + O(\|x\|^{-1}) \text{ as } \|x\| \to \infty.
\end{align*}
Define functions $\alpha, \beta: \Rn \setminus \{0\} \to\RL$ as
$$
\alpha(v) = c^Tv - \sqrt{v^TMv} \mbox{ and } \beta(v) = 
d + \frac{v^T M x_*}{\sqrt{v^TMv}} .
$$
For any $v \neq 0$ the function $f$ satisfies
$f(v t) = \alpha(v) t + \beta(v) + O(1/t) \text{ as } t \to \infty$.
This states that the function $f$, restricted to a ray from the origin,
has a slant asymptote.

We can easily determine the unique values of the parameters $c$ and $d$ 
for a given SOCF $f$ from the values of $\alpha$ and $\beta$ evaluated
at standard basis vectors $e_i \in \Rn$.  The components of $c$ and the scalar $d$ are
$$
 e_i^T c = \frac{\alpha(e_i) + \alpha(-e_i)} 2, \quad
d = \frac {\beta(e_1) + \beta(-e_1)} 2.
$$
The matrix $M$ is uniquely determined by the $\alpha$ function.  
Let $0 < \lam_1 \leq \lam_2 \cdots \leq \lam_n$ 
be the eigenvalues of $M$.  Then 
$$
v^TMv = \frac {(\alpha(v) + \alpha(-v))^2}{4}, \text{ so }
\lam_n =\max_{v \neq 0} \frac {(\alpha(v) + \alpha(-v))^2}{4 v^T v}.
$$
Any maximizing vector is an eigenvector corresponding to $\lam_n$.
Then restrict to the subspace orthogonal to that eigenvector, and compute
$\lam_{n-1}$ and a corresponding eigenvector.
In this way all the eigenvalues of $M$ and a complete set of 
orthogonal eigenvectors can be computed, and thus the matrix $M$ is
uniquely determined.

After $M$ is known, we can compute $x_*$, since
$$
\beta(M^{-1} e_i) = d + \frac{e_i^T x_*}{\sqrt{e_i^T M^{-1} e_i}}
$$
and thus the components of $x_*$ are
$$
e_i^T x_* = \frac{\beta(M^{-1} e_i)  - \beta(-M^{-1} e_i)}{2} \sqrt{e_i^T M^{-1} e_i }.$$
The value of $f$ at any point determines the final parameter, $\delta$. The most convenient input is $x_*$, in which case
$\delta = c^T x_*  + d - f(x_*)$.
}
} % end of commented out old proof
\end{proof}
\begin{example}
%This illustrates the first part of Theorem~\ref{uniqueParams}.
Let $f(x_1, x_2) = -\sqrt{4 + (x_1-1)^2}$ 
be the SOCF on $\RL^2$ defined by 
$c = 0$, $d = 0$, $\delta = 2$,
$M = \left [ \begin{smallmatrix} 1 & 0 \\ 0 & 0 \end{smallmatrix} \right ]$, 
and $x_* = (1,0)$.
%then $x_* = (0,0)$ and $x_* = (0, 1)$
%$x_* = \left [ \begin{smallmatrix} 0 \\ 0 \end{smallmatrix} \right ]$ and 
%$x_* = \left [ \begin{smallmatrix} 0 \\ 1 \end{smallmatrix} \right ]$ 
Note that $M$ is not positive definite.
The null space of $M$ is $\text{span}\{(0,1)\}$.
The parameterization is not unique since any $x_* \in \{(1, a) \mid a \in \RL\}$
yields the same SOCF.  
\end{example}
While many choices of $A$ and $b$ in the form of Equation \eqref{socfOld} yield the same function, there is a canonical
choice for $A$ and $b$ starting with the function in the form of Equation \eqref{socf}.
Recall that a positive semidefinite matrix $M$ has a unique positive semidefinite square root, denoted $M^{1/2}$.  

\begin{theorem}
Let $M \in \RL^{n\times n}$ be positive semidefinite, $x_* \in \Rn$, and $\delta \in \RL$.
Then
$\delta^2 + (x-x_*)^T M (x-x_*)= \| A x + b \|^2$
for 
$$
A = \begin{bmatrix} M^{1/2} \\ 0  \end{bmatrix},\text{ and } b = \begin{bmatrix} - M^{1/2} x_* \\ \delta  \end{bmatrix}.
$$
The last row of $A \in \RL^{(n+1) \times n}$ is all 0s, and the last component of $b\in \RL^{n+1}$ is $\delta$.
\end{theorem}
\begin{proof}
Note that $M^{1/2}$ is symmetric, and
$$
A x + b = \begin{bmatrix} M^{1/2} x \\ 0  \end{bmatrix} + \begin{bmatrix} -M^{1/2} x_* \\ \delta  \end{bmatrix} =
\begin{bmatrix} M^{1/2} (x-x_*) \\ \delta  \end{bmatrix}.
$$ 
Thus $\| A x + b \|^2 = \delta ^2 + (x-x_*)^T M^{1/2} M^{1/2}(x-x_*) = 
\delta^2 + (x-x_*)^T M (x-x^*)$.
\end{proof}
\begin{remark}
It follows from this theorem that any SOCF can be defined in the form of Equation \eqref{socfOld} with $A \in \RL^{(n+1)\times n}$.  
While $A$ is an $m\times n$ matrix with any $m$, 
using $m > n+1$ is
never needed.
\end{remark}
Recall that any non-constant SOCF is not bounded below, since it is concave.  
We give necessary and sufficient conditions for an SOCF to be bounded above with two theorems.  
The next theorem assumes that $M$ is positive definite, 
and the case where $M$ is positive semidefinite is handled in Theorem~\ref{fBoundedSemiDef}.
\begin{theorem}
\label{fBoundedDef}
The SOCF $f: \Rn \to \RL$ written in the form (\ref{socf}),
$$f(x) = c^T x + d - \sqrt{\delta^2 + (x-x_*)^T M (x-x_*)}, $$ 
with $M$ positive definite, is bounded above if and only if $c^T M^{-1}c \leq 1$.  More specifically,
%The SOCF $f$ is not bounded below, and $f$ is bounded above if and only if $c^T M^{-1}c \leq 1$.  Furthermore,
\begin{enumerate} % We can go back to the labels \item[(a)] 
\item if $c^T M^{-1}c < 1$ and $\delta = 0$, 
then $x_*$ is the unique critical point of $f$, 
and $f(x_*) = c^T x_* + d$ is the global maximum value of $f$,

\item if $c^T M^{-1} c =  1$ and $\delta = 0$, 
then every point in the ray 
$\{x_* + t M^{-1}c \mid t \geq 0\}$ is a critical
point of $f$, on which $f$ attains its maximum value of $f(x_*) = c^T x_* + d$, 

\item if $c^T M^{-1} c >  1$ and $\delta = 0$, 
then $x_*$ is the unique critical point of $f$, but $f$ is not bounded above, 

\item if $c^T M^{-1}c < 1$ and $\delta > 0$, 
then $\displaystyle x_{cp} := x_* + \frac {\delta M^{-1} c }{\sqrt{1 - c^T M^{-1} c}}$ is the unique critical point of $f$, and
$f(x_{cp}) = c^T x_{cp} + d - \delta \sqrt{1-c^TM^{-1}c}$ is the global maximum value of $f$,

\item if $c^T M^{-1} c = 1$ and $\delta  > 0$, 
then $f$ has no critical points and $f$ does not have a global maximum value, but $f$ is bounded above by $c^T x_* + d$, and

\item if $c^T M^{-1} c > 1$ and $\delta  > 0$, 
then $f$ has no critical points and $f$ is not bounded above.
%$f$ achieves its global maximum value of $f(x_*) = d$ 
%at each point in the ray $\{x_* + t M^{-1}c \mid t \geq 0\}$, and
\end{enumerate}
\end{theorem}
\begin{proof}
To simplify the proof, we will analyze the SOCF 
$\tf(x) := c^T x - \sqrt{\delta^2 + x^T M x}$.
Note that 
$f(x+ x_*) = \tf(x) + (d + c^T x_*)$, and $\tf(x - x_*) = f(x) -d + c^T x_*$,
so we can easily relate the critical points, and the upper bounds, of $f$ and $\tf$.

{\bf Case I:} $\delta = 0$.
In this case $\tf(x) = c^T x - \sqrt{x^TMx}$. 
Let $v$ be any nonzero vector in $\Rn$.  Since $v^T Mv > 0$, the function $t \mapsto \tf(vt) = (c^T v)t - \sqrt{v^TM v} \, |t|$ is not differentiable at $t = 0$.  Thus  $0 \in \RL^n$ is a critical point of $\tf$ at which $\tf$ is not differentiable, and $f$ has a critical point at $x_*$.
To determine if $\tf$ has a global maximum at 0, define $g_v: [0, \infty) \to \RL$ by $g_v(t) = \tf(v t) = (c^T v)t - t \sqrt{v^T M v}$.
Note that $g_v$ is a linear function giving the value of $\tf$ 
along a ray starting at $0 \in \Rn$ with the direction vector $v$.
The function $\tf$ is bounded above if and only if the slope of $g_v$ is non-positive 
for all directions $v$.

Let $\mathcal E = \{ x \in \Rn \mid x^T M x = 1\}$. 
Note that $\mathcal E$ is an ellipsoid centered at $0$, since $M$ is positive definite.  
Furthermore, $g_v(0) = 0$, so $\tf$ is bounded above if and only if
the maximum value of $\tf$, restricted to $\mathcal E$, is non-positive.  
We compute this maximum value using the method
of Lagrange multipliers.  The extreme values of $\tf$ restricted to $\mathcal E$ occur at 
places where $\nabla (c^T x) = \lambda \nabla \sqrt{x^T M x}$.
This is equivalent to $c = \lambda Mx/\sqrt{x^T M x}$, or $\lambda M x = c$ since $x^T M x = 1$ on $\mathcal E$.
Thus, the extrema of $\tf$ are at $x = \frac 1 \lam M^{-1}c$, where $\lambda$ is determined by $x^T Mx = 1$.  Thus
$\frac 1 {\lam^2} c^T M^{-1} M M c = 1$, so $\lam^2 = c^T M^{-1} c$. 
There are two antipodal points on 
$\mathcal E$, $x_{\pm} = \frac{\pm 1}{ \sqrt{c^T M^{-1} c} } M^{-1} c$, 
with extreme values of $\tf$ restricted to $\mathcal E$.
We see that $\tf(x_{\pm}) = c^T x_{\pm} - 1 = \pm \sqrt{c^T M^{-1}c} - 1$.  
The maximum value of $\tf$ restricted to $\mathcal E$ is $\sqrt{c^T M^{-1}c} - 1$,
which occurs at $x_+$.
Thus, the maximum slope of $g_v$ occurs when $v$ is a positive scalar multiple 
of $M^{-1} c$, and that maximum slope has the same sign as $\sqrt{c^T M^{-1}c} - 1$,  
Thus, $\tf$ is bounded above if and only if $c^T M^{-1} c \leq 1$.

If $c^T M^{-1}c < 1$, then $0$ is the unique critical point of $\tf$, 
and $\tf(0) = 0$ is the global maximum value of $\tf$. 
Thus $x_*$ is the unique critical point of $f$, 
and $f(x_*) = c^T x_* + d$ is the global maximum value of $f$.
This proves part 1 in the theorem.
If $c^T M^{-1} c = 1$, then the linear function $g_v$ has slope $0$ when $v = M^{-1}c$, and 
$\tf$ achieves its maximum value of 0 
at each point on the ray from $0$ through $M^{-1}c$.
Every point in this ray, 
$\mathcal C = \{t M^{-1} c \mid t \geq 0\}$, is a critical point. 
Translating this result to the original $f$ proves part 2.
If $c^T M^{-1}c > 1$, then the slope of $g_v$ is positive for some $v$.  Thus $\tf$ has an isolated critical point at $0$, but $\tf$ is not bounded above. This proves part 3 of the theorem.

{\bf Case II:} $\delta > 0$.
The gradient of $\tf$ at $x$ is 
%
%We will prove the result with $x_* = 0$ and the theorem follows by
%adding $x_*$ where appropriate.
%The gradient of $\tf$ defined by 
%$\tf(x) = c^T x - \sqrt{\delta^2 + x^T M x}$ is 
$$
\nabla \tf (x) =  c - \frac{M x}{\sqrt{\delta^2 + x^T M x}}.
$$
In this case $\tf$ is smooth, and
the critical points of $\tf$ are solutions to $\nabla \tf (x) = 0$.
Since $M$ is positive definite, $\tf$ is strictly concave 
by Theorem~\ref{strictlyConcave} and $\tf$ has at most one critical point.  
If $\tf$ has a critical point
then it must be a global maximum and hence $\tf$ is bounded above.
%, which is equivalent to $x = \sqrt{\delta^2 + x^T M x} \, (M^{-1} c)$. 
Denote the critical point of $\tf$ as $x_{cp}$, if it exists, satisfies 
$Mx_{cp} = c \sqrt{\delta^2 + {x_{cp}}^T M x_{cp}}$. It follows that the critical point is a scalar multiple of $M^{-1}c$.  
Let $x_{cp}= \alpha M^{-1} c$. The scalar $\alpha$ satisfies 
$\alpha = \sqrt{\delta^2 + \alpha^2(M^{-1}c)^T M (M^{-1} c)} = \sqrt{\delta^2 + \alpha^2 c^T M^{-1} c}$.
If $c^T M^{-1} c < 1$ the unique solution is
$\alpha_s = \delta / \sqrt{1 - c^T M^{-1} c}$, 
and if $c^T M^{-1} c \geq 1$ there are no solutions for $\alpha$.  
Thus if $c^T M^{-1} c < 1$ the function $\tf$ has the critical point $\alpha_s M^{-1}c$,
and the critical point of $f$ is $x_{cp} = x_* + \alpha_s M^{-1}c$, 
and a calculation of 
%the value of 
$f(x_{cp})$ completes the proof of part 4.

We have already seen that $\tf$, and therefore $f$, 
has no critical points when $c^T M^{-1}c \geq 1$.  
The results about boundedness and upper bounds need the following asymptotic analysis.
When $\|x\|$ is large
then $x^T M x$ is large of order $O(\|x\|^2)$ because $M$ is positive definite, and
$\sqrt{\delta^2 + x^T M x} = \sqrt{x^T M x}\sqrt{1 + \delta^2/(x^T M x)} > \sqrt{x^T M x}$.
The Taylor expansion $\sqrt{1 + \eps} = 1 + \eps/2 + O(\eps^2)$ shows that
$\sqrt{\delta^2 + x^T M x }= \sqrt{x^T M x}\big (1 + \delta^2/(2 x^T M x) + O(\|x\|^{-4} \big )$.
Thus, an SOCF with $\delta >0$ is always less than the corresponding SOCF with $\delta = 0$,
and the difference approaches 0 as $\|x\| \to \infty$.
Parts 5 and 6 of the theorem follow.
\end{proof}

\begin{figure}
% These figures were written by "SOCF contours2023.nb"
\begin{center}
    \begin{tikzpicture} 
        \node at (0.2, 2.3) {$c^T M^{-1} c < 1$};
        \node at (4.7, 2.3) {$c^T M^{-1} c = 1$};
        \node at (9.2, 2.3) {$c^T M^{-1} c > 1$};
        \node at (-2.6, 0.17) {$\delta = 0$};
        \node at (-2.6, -4.33) {$\delta = 1$};
        \node at (0,0) {\includegraphics[width = 4cm]{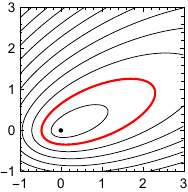}};
        \node at (4.5,0) {\includegraphics[width = 4cm]{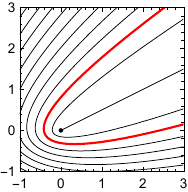}};
        \node at (9,0) {\includegraphics[width = 4cm]{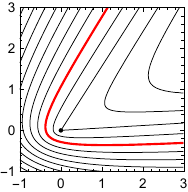}};
        \node at (0,-4.5) {\includegraphics[width = 4cm]{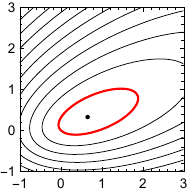}};
        \node at (4.5,-4.5) {\includegraphics[width = 4cm]{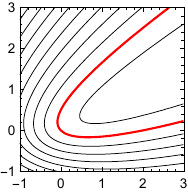}};
        \node at (9,-4.5) {\includegraphics[width = 4cm]{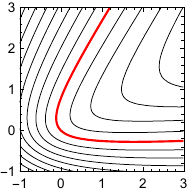}};
    \end{tikzpicture}
\end{center}
\caption{
\label{SOCF6}
Contour plots of six different second-order cone functions
defined in Example~\ref{ex:SOCF6}.
All functions have the same positive definite matrix $M$.
The parameters $c$ and $\delta$ are chosen to illustrate Theorem~\ref{fBoundedDef}, which says that 
$f$ is bounded above if and only if $c^T M^{-1}c \leq 1$.
The six parts of Theorem~\ref{fBoundedDef} correspond to the six contour plots.
The contour with $f(x) = -1$ is a thick red curve, and
the spacing between contours is
$\Delta f = 0.5$.
}
\end{figure}
\begin{example}
\label{ex:SOCF6}
    Figure \ref{SOCF6} shows the contour diagrams of 6 SOCFs of the form
    $f(x)= c^Tx - \sqrt{\delta^2 + x^T M x},$
    with $M = \left[\begin{smallmatrix}2 &-1\\-1 & 5 \end{smallmatrix} \right ]$. 
    The eigenvalues of $M$ are $(7 \pm \sqrt{13})/2$, so
    $M$ is positive definite and Theorem~\ref{fBoundedDef} applies with the parameters $d = 0$ and $x_* = (0, 0)$.
    A calculation shows that $M^{-1} = \frac 1 9 \left[\begin{smallmatrix}5 &1\\1 & 2 \end{smallmatrix} \right ]$.
    The other parameters are $\delta = 0$ in the top row, $\delta = 1$ in the bottom row,
    $c = (0.7, 0.7)$ in the left column, $c = (1,1)$ in the middle column, and $c = (1.3, 1.3)$ in the right column.
    These values of $c$ give $c^TM^{-1}c = 0.7^2, 1$, and $1.3^2$, respectively.
    In the top row $(0, 0)$ is always a critical point and $f(0,0) = 0$.
    In the top middle figure
    %figure with $c^TM^{-1}c = 1$ and $\delta = 0$, 
    the contour with height 0
    is the ray in the direction $M^{-1}c = \big (\frac 2 3, \frac 1 3 \big )$.
    In the 
    %figure with $c^TM^{-1}c = 0.7^2$ and $\delta = 1$, 
    bottom left figure
    we find that 
    $x_{cp} = \big (\frac {1.4} 3, \frac {0.7} 3 \big )/\sqrt{1-.7^2}\approx (0.65, 0.33)$ and $f(x_{cp}) = - \sqrt{1-.7^2} \approx -.71$. 
    In the bottom middle figure $f$ is bounded above by 0.
\end{example}

\begin{theorem}
\label{fBoundedSemiDef}
The SOCF $f: \Rn \to \RL$ written in the form (\ref{socf}), $$f(x) = c^T x + d - \sqrt{\delta^2 + (x-x_*)^T M (x-x_*)},$$ 
with $M$ positive semidefinite
is bounded above if and only if $c \in \text{col}(M)$ and $c^T M^+ c \leq 1$.
\end{theorem}
\begin{proof}
Since $M$ is symmetric, the Fundamental Theorem of Linear Algebra states that the null space of $M$
is the orthogonal complement of the 
column space of $M$.
We can write any $x \in \Rn$ as $x = x_n + x_r$ for unique $x_n \in N(M)$ and $x_r \in \text{col}(M)$. 
%We use $x_r$ to suggest the ``range'' of $M$, to avoid confusion when we
Similarly, we split $c = c_n + c_r$.
For a fixed $x_*$, we write $x \in \Rn$ as $x = x_* + x_n + x_r$.
Thus $M(x-x*) = M(x_n + x_r) = M x_r$,
and
the general second-order cone function is
$$
f(x_* + x_n + x_r) = c_n^T x_n + c_r^T x_r + (c^T x_* + d) - \sqrt{\delta^2 + {x_r}^T M x_r}.
$$

Assume $c \not \in \text{col}(M)$.  Then $c_n \neq 0$ and $f$ is not bounded since 
$f(x_* + x_n) = c_n^T x_n + (c^T x_* + d - \delta)$ 
is an unbounded linear function.

Assume $c \in \text{col}(M)$, so $c_n = 0$.  Define $g: \text{col}(M) \to \RL$ by
$$
g(x_r) := f(x_* + x_n + x_r) = c_r^T x_r + (c^T x_* + d) - \sqrt{\delta^2 + {x_r}^T M x_r}.
$$
Note that $M$, restricted to $\text{col}(M)$ is a nonsingular map, 
so we can apply Theorem~\ref{fBoundedDef} to $g$ as follows. The pseudo-inverse of $M$ satisfies
$$
M^+(x_n + x_r) = M^+ x_r, \quad MM^+ (x_n + x_r) = M^+ M(x_n + x_r) = x_r.
$$
Thus, the restriction of $M^+$ to $\text{col}(M)$ is the inverse of the restriction of $M$ to $\text{col}(M)$.
Theorem~\ref{fBoundedDef} says that $g$ is bounded above if and only if 
$c_r^T M^+c_r \leq 1$.
Note that $c^T M^+ c = c_r^T M^+ c_r$ for any $c \in \Rn$.

We have shown that $f$ is not bounded above if $c \not \in \text{col}(M)$.
We have also shown that if $c \in \text{col}(M)$ then $f$ is bounded above
if and only if $c^T M^+ c \leq 1$.
These two statements can be combined into one: 
$f$ is bounded above if and only if 
$c\in \text{col}(M)$ and $c^T M^+ c \leq 1$.
%The theorem follows.
\end{proof}
\begin{remark}
    If $M$ is positive definite, then $c \in \text{col}(M) = \Rn$ and $M^+ = M^{-1}$. 
    Thus Theorem~\ref{fBoundedSemiDef}, in the case where $M$ is positive definite, implies the that $f$ is bounded above if and only if $c^T M^{-1}c \leq 1$, which is the first part of Theorem~\ref{fBoundedDef}.
\end{remark}
\begin{example}
    Consider the SOCF on $\RL^2$ with 
    $M = \left[\begin{smallmatrix}4 &0\\0 & 0 \end{smallmatrix} \right ]$, $d = 0$, and $x_* = (0,0)$,
    $$f(x) = c_1 x_1 + c_2 x_2 - \sqrt{\delta^2 + 4 x_1^2}.$$
    
    Note that $f(0, x_2) = c_2 x_2 - \delta$ is not bounded above if $c_2 \neq 0$.  If $c_2 = 0$ then $f(x) = c_1 x_1 - \sqrt{\delta^2 + 4 x_1^2} \leq c_1 x_1 - 2 |x_1|$ and $f(x) \to c_1 x_1 - 2 |x_1|$ as $x_1 \to \pm \infty$.  Thus, $f$ is bounded above if and only if
    $c_2 = 0$ and $c_1^2 \leq 4$.

    This observation is predicted by Theorem~\ref{fBoundedSemiDef}.
    The column space of $M$ is $\text{col}(M) = \{(a,0) \mid a \in \RL\}$, so $c \in \text{col}(M)$ is equivalent to $c_2 = 0$.
    The pseudo-inverse of $M$ is 
    $M^{+} = \frac 1 4 \left[\begin{smallmatrix}1 &0\\0 & 0 \end{smallmatrix} \right ]$, so $c^T M^{+}c = c_1^2/4$, and $c^T M^{-1} c \leq 1$ is equivalent to
    $c_1^2 \leq 4$.
\end{example}

One of that main uses of SOCFs is to define convex sets for optimization problems.
Optimization over a bounded set is very different from optimization over an unbounded set,
so we finish this paper with a simple characterization.
\begin{theorem}
    Let $f: \Rn \to \RL$ be defined by 
    $f(x) = c^T x + d - \sqrt{\delta^2 + (x-x_*)^T M (x-x^*)}$, where $M$ is positive semidefinite.  
    Let $\R := \{x \in \Rn \mid f(x) \geq 0 \}$, and assume 
    $\R \neq \emptyset$.
    The feasible region
    $\R$
    is closed and convex.  Furthermore, $\R$ is bounded if and only if $M$ is positive
    definite and $c^T M^{-1} c < 1$.
    \end{theorem}
\begin{proof}
    The set $\R$ is convex since $f$ is concave, and it is closed since $f$ is continuous.  We now prove the last statement of the theorem by determining whether $\R$ is bounded or unbounded for any SOCF.
 
  First, assume $M$ is not positive definite and $c \not \in \text{col}(M)$.  Theorem~\ref{fBoundedSemiDef} says that $f$ is not bounded above and hence $\R$ is not bounded.
  
  Second, assume $M$ is not positive definite and $c \in \text{col}(M)$.  The proof of Theorem~\ref{fBoundedSemiDef} shows that 
  $f(x+ x_n) = f(x)$ for all $x_n \in N(M)$. Choose $\tilde x \in \R$, which is possible since $\R \neq \emptyset$.  The affine subspace $\tilde x + N(M)$ is a subset of $\R$, so $\R$ is unbounded.

Finally,
assume $M$ is positive definite, and thus $M^{-1}$ exists.  The nonempty $\R$ is bounded if and only if conditions 1 or 4 of Theorem~\ref{fBoundedDef} are satisfied. 
Thus, $\R$ is bounded if and only if $c^T M^{-1} c < 1$. 
See Figure \ref{SOCF6} for examples.  
% since we have shown when $\R$ is bounded for every SOCF.
\end{proof}
\begin{remark}
    In the case where $M$ is positive definite and $c^T M^{-1} c < 1$, the compact $\R$ might be trivial.  Let $\tilde d := d - \delta \sqrt{1-c^TM^{-1} c}$.  The set $\R$ is the empty set if $\tilde d < 0$, $\R$ is the singleton set $\{x_{cp}\}$ if $\tilde d = 0$,
    and $\R$ has a non-empty interior if $\tilde d > 0$.
\end{remark}
\setcounter{equation}{0}
\section{Conclusion}
The second-order cone function has important application in optimization problems. Our work gives
necessary and sufficient conditions for strict concavity of a second-order cone function. We show that every SOCF can be written in the 
%seemly more convenient 
form
$f(x) = c^T x + d -\sqrt{\delta^2 + (x-x_*)^TM(x-x_*)}$,
which has unique parameters in many cases.
This alternative parameterization gives a deep understanding of the family
of SOCFs.
This alternative description leads to new results about SOCFs.
We characterize the critical points and global maxima of $f$, depending
on the parameters.
We give necessary and sufficient conditions for $f$ to be bounded above,
and for the set $\{x \in \Rn \mid f(x) \geq 0\}$ to be bounded.
Our results can lead to improved algorithms for optimization problems involving second-order cone constraints.

\section*{Funding Statement, Data Availability Statement, Conflicts of Interest, and Acknowledgements}

This research was performed as part of the employment 
of the authors at Northern Arizona University.
The authors received no specific funding for this research.
The data used to support the findings of this study are included within the article.
The authors declare that there is no conflict of interest regarding the publication of this article.
The authors thank the Cornell University arXiv for posting a
preprint of this article \cite{arXiv}.

\end{document}